\documentclass[a4,10pt]{amsart}
\usepackage{amsmath, amssymb}
\usepackage{amsthm}
\usepackage{amsfonts}

\usepackage{latexsym}
\usepackage{mathrsfs}

\usepackage{geometry}
\usepackage{dsfont}
\usepackage{hyperref}
\usepackage[dvips]{graphicx}
\usepackage{color}
\usepackage[all]{xy}

\date{\today}
\allowdisplaybreaks

\theoremstyle{plain}
\newtheorem{theorem}{Theorem}[section]
\newtheorem{proposition}[theorem]{Proposition}
\newtheorem{lemma}[theorem]{Lemma}
\newtheorem{corollary}[theorem]{Corollary}
\newtheorem{example}[theorem]{Example}
\newtheorem*{open question}{Open Question}
\newtheorem{definition}[theorem]{Definition}

\theoremstyle{definition}

\theoremstyle{remark}

\newcommand{\Spec}{\operatorname{Spec}}
\newcommand{\wSpec}{w\mbox{-Spec}}

\def\GV{{\rm GV}}
\def\tor{{\rm tor_{\rm GV}}}
\def\Hom{{\rm Hom}}
\def\Ext{{\rm Ext}}

\def\GV{{\rm GV}}
\def\Max{{\rm Max}}
\def\DW{{\rm DW}}

\def\Spec{{\rm Spec}}
\def\Max{{\rm Max}}
\begin{document}

\title[Rings with uniformly  $S$-$w$-Noetherian spectrum]{Rings with uniformly  $S$-$w$-Noetherian spectrum}

\author[X. Zhang]{Xiaolei Zhang}
\address{School of Mathematics and Statistics, Tianshui Normal University, tianshui 741001, China}
\email{zxlrghj@163.com}

\keywords{uniformly $S$-$w$-Noetherian spectrum, ascending chain condition, countably generated ideal, polynomial ring}
\subjclass[2020]{13A15, 13E99}

\begin{abstract} 
In this paper, the notion of rings with uniformly  $S$-$w$-Noetherian spectrum is introduced. Several characterizations of rings with uniformly  $S$-$w$-Noetherian spectrum are given. Actually, we show that a ring
$R$ has uniformly $S$-$w$-Noetherian spectrum with respect to some $s\in S$ if and only if  each ascending chain of radical $w$-ideals of $R$ is stationary   with respect to $s\in S$, if and only if each radical (prime) \mbox{$(w$-$)$ideal} of $R$ is radically $S$-$w$-finite with respect to $s$, if and only if  each  countably generated ideal  of $R$ is radically $S$-$w$-finite with respect to $s$, if and only if  $R[X]$ has  uniformly $S$-$w$-Noetherian spectrum, if and only if
$R\{X\}$ has  uniformly $S$-Noetherian spectrum.

Beside, we show a ring $R$ has  Noetherian (resp., uniformly  $S$-Noetherian) spectrum with respect to $s$ if and only if each countably generated ideal of $R$ is radically finite ($S$-finite with respect to $s$), which is a new result in the classical case.
\end{abstract}

\maketitle

\section{Introduction}
Throughout this paper, $R$ is always  a commutative ring with identity. For a subset $U$ of  an $R$-module $M$, we denote by $\langle U\rangle$ the $R$-submodule of $M$ generated by $U$. A subset $S$ of $R$ is called a multiplicative subset of $R$ if $1\in S$ and $s_1s_2\in S$ for any $s_1\in S$, $s_2\in S$.

Early in 1968,  Ohm and Pendleton \cite{OP68} introduced a ring $R$ with \emph{Noetherian spectrum}, i.e., the prime spectrum $\Spec(R)$ of $R$ is Noetherian. Since the open subsets of  $\Spec(R)$ and the radical ideals of $R$ correspond to each other one-to-one in  order,  a ring $R$ has the Noetherian spectrum if and only if the ascending
chain condition holds for the radical ideals of $R$.  In 1985, Ribenboim \cite{r85} showed that showed that a ring $R$ has the Noetherian spectrum if and only if every (prime) ideal of $R$ is \emph{radically of finite type} (or radically finite), i.e., an ideal $I$  satisfies that $\sqrt{I}=\sqrt{J}$ for some finitely generated subideal $J$ of $I$, where $\sqrt{I}:=\{r\in R\mid r^n\in I\ \mbox{for some}\ n\geq1\}$. Some more works on rings with Noetherian spectrum  can be found in \cite{B01,fhp97,TLC13}.

In 2014, Kim, Kwon and Rhee \cite{KKR14}  called an integral domain $R$ having \emph{strong Mori spectrum} (or \emph{$w$-Noetherian spectrum})
if the induced topology on $\wSpec(R)$ by the Zariski topology on $\Spec(R)$ is Noetherian, or equivalently, it satisfies the descending chain condition on the sets of the form $
W(I):=\{\,P\in\wSpec(R)\mid I\subseteq P\,\}$
where $I$ runs over $w$-ideals of $R$. Moreover, several characterizations of integral domains with strong Mori spectrum were given (see \cite[Theorem 2.6]{KKR14}). Recently, Samir and Ahmed \cite{GH25}  led into the notion of rings with uniformly $S$-Noetherian spectrum. An ideal $I$ of $R$ is said to be \emph{radically $S$-finite} (with respect to $s$) if $sI \subseteq \sqrt{J}$ for some finitely generated subideal $J$ of $I$ and some $s \in S$. A ring $R$ is said to have an \emph{uniformly $S$-Noetherian spectrum} if there is $s\in S$ such that every ideal of $R$ is radically $S$-finite with respect to $s$. Similar characterizations of rings with uniformly $S$-Noetherian spectrum to classical cases can be found in \cite[Theorem 3.11, Proposition 3.16]{GH25}. Subsequently, the authors  \cite{ZAK} investigated  rings with uniforly $S$-Noetherian spectrum  under various constructions,including flat overrings, Serre’s conjecture rings, Nagata rings, Anderson rings, and semigroup rings.

Since the $w$-operations and multiplicative subsets play a different methods to study rings and modules. It is very interesting to unify them together. For example, the  uniformly $S$-$w$-Noetherian rings introduced in  \cite{zuswn-24} are generalizations of both $w$-Noetherian rings and uniformly $S$-Noetherian rings. The main purpose of this paper is to generalize  integral domains with strong Mori spectrum and rings with uniformly $S$-Noetherian spectrum. Actually, we introduce the notion of rings with uniformly $S$-$w$-Noetherian spectrum. An ideal $I$ of $R$ is said to be \emph{radically $S$-$w$-finite } (with respect to $s$) if $sI \subseteq \sqrt{F_w}$ for some finitely generated subideal $F$ of $I$ and some $s \in S$.
A ring $R$ is said to have \emph{uniformly  $S$-$w$-Noetherian spectrum}  (with respect to $s$) if there exists $s\in S$ such that every ideal of $R$ is radically $S$-$w$-finite  with respect to $s$. We characterize rings with uniformly  $S$-$w$-Noetherian spectrum via various aspects:
\begin{theorem}$($=Theorem \ref{main-1}+Theorem \ref{main-2}$)$
	Let $R$ be a ring, $S$ a multiplicative subset of $R$ and $s\in S$.
	Then the following statements are equivalent.
	\begin{enumerate}
		\item $R$ has uniformly $S$-$w$-Noetherian spectrum with respect to $s$.
		\item Each ascending chain of radical $w$-ideals of $R$ is stationary   with respect to $s\in S$.
		\item Each radical \mbox{$(w$-$)$ideal} of $R$ is radically $S$-$w$-finite with respect to $s$.
		\item Each prime  \mbox{$(w$-$)$ideal} of $R$ is radically $S$-$w$-finite with respect to $s$.
		\item Each  countably generated ideal  of $R$ is radically $S$-$w$-finite with respect to $s$.
		\item $R[X]$ has  uniformly $S$-$w$-Noetherian spectrum with respect to $s$.
		
		\item $R\{X\}$ has  uniformly $S$-Noetherian spectrum  with respect to $s$.
	\end{enumerate}	
\end{theorem}
Beside, we show a ring $R$ has  Noetherian (resp., uniformly  $S$-Noetherian) spectrum with respect to $s$ if and only if each countably generated ideal of $R$ is radically finite ($S$-finite with respect to $s$), which is a new result in the classical case.

As our work involves $w$-operations, we recall the related notions from \cite{FK24}. Let $R$ be a ring and $J$ a finitely generated ideal of $R$. Then $J$ is called a \emph{$\GV$-ideal} if the natural homomorphism $R\rightarrow \Hom_R(J,R)$ is an isomorphism. The set of all $\GV$-ideals is denoted by $\GV(R)$. Let $M$ be an $R$-module. Define

\begin{center}
	{\rm $\tor(M):=\{x\in M|Jx=0$, for some $J\in \GV(R) \}.$}
\end{center}
An $R$-module $M$ is said to be \emph{$\GV$-torsion} (resp. \emph{$\GV$-torsion-free}) if $\tor(M)=M$ (resp. $\tor(M)=0$). A $\GV$-torsion-free module $M$ is called a \emph{$w$-module} if $\Ext_R^1(R/J,M)=0$ for any $J\in \GV(R)$, and the \emph{$w$-envelope} of $M$ is given by
\begin{center}
	{\rm $M_w:=\{x\in E(M)|Jx\subseteq M$, for some $J\in \GV(R) \},$}
\end{center}
where $E(M)$ is the injective envelope of $M$. Therefore, $M$ is a $w$-module if and only if $M_w=M$.  A \emph{$\DW$ ring} $R$ is a ring for which every $R$-module is a $w$-module. Examples of $\DW$ rings include semihereditary rings and commutative rings with krull dimension $\leq 1$.
The set of all prime $w$-ideals is denoted by $w$-$\Spec(R)$.
A \emph{maximal $w$-ideal} is an ideal of $R$ which is maximal among the $w$-submodules of $R$. The set of all maximal $w$-ideals is denoted by $w$-$\Max(R)$. All maximal $w$-ideals are  prime ideals , i.e., $w$-$\Max(R)\subseteq w$-$\Spec(R)$ (see {\cite[Theorem 6.2.14]{FK24}}).

\section{Rings with uniformly  $S$-Noetherian spectrum}

Recently, Samir and Hamed \cite{GH25} introduced the notion of rings with uniformly $S$-Noetherian spectrum.
An ideal $I$ of $R$ as being radically $S$-finite (with respect to $s$) if $sI \subseteq \sqrt{J} \subseteq \sqrt{I}$ for some finitely generated subideal $J$ of $I$ and some $s \in S$. A ring $R$ is said to have an \emph{uniforly $S$-Noetherian spectrum} (abbreviated to be $u$-$S$-Noetherian spectrum) if there is $s\in S$ such that every ideal of $R$ is radically $S$-finite with respect to $s$.

They gave several characterizations of rings with uniforly $S$-Noetherian spectrum. Actually, they showed that 
\begin{theorem} \cite[Theorem 3.11, Proposition 3.16]{GH25} Let $R$ be a commutative ring, $S$ a  multiplicative subset of $R$ and $s\in S$. Then the following statements are equivalent.
	\begin{enumerate}
		\item $R$ has uniformly  $S$-Noetherian spectrum with respect to $s$.
		\item  Each ascending chain of radical ideals of $R$ is stationary   with respect to  $s\in S$.
		\item  Each prime $($radical$)$ ideal of $R$ is radically finite with respect to $s\in S$.
	\end{enumerate}
\end{theorem}

 It is well-known that a ring $R$ is a Noetherian ring (resp.,  PIR)  if and only if every countably generated ideal of $R$ is finitely $($resp., principally$)$ generated.
The author in \cite{zuswn-24} obtained that a ring $R$ is a uniformly $S$-Noetherian ring (resp.,  uniformly $S$-PIR) if and only if there is $s\in S$ such that  every countably generated ideal of $R$ is $S$-finite $($resp., $S$-principal$)$ with respect to $s$ under a mild condition. The main aim in this section is to extend this result to rings with uniformly $S$-Noetherian spectrum.

\begin{theorem} \label{count-us}
	Let $R$ be a commutative ring, $S$ a  multiplicative subset of $R$ and $s\in S$. Then  $R$ has uniformly  $S$-Noetherian spectrum with respect to $s$ if and only if each countably generated ideal of $R$ is radically $S$-finite with respect to $s$.
\end{theorem}
\begin{proof}  Suppose $R$ has uniformly  $S$-Noetherian spectrum with respect to $s$. Then, trivially, each countably generated ideal of $R$ is radically $S$-finite with respect to $s$.
	
On the other hand, let $I=\langle r_1,r_2,\cdots\rangle$ be a countably generated ideal of $R$. Set  $I_n=\sqrt{\langle r_1,r_2,\cdots,r_n\rangle}$. Then $\{I_n\mid n\geq 1\}$ is an ascending chain of radical ideals of $R$. So there is a positive integer $k$ such that $sI_n\subseteq I_k$ for all $n\geq k$. Note that $\sqrt{I}=\bigcup\limits_{n\geq 1}I_n$. Indeed, for any $r\in \sqrt{I}$, we have $r^m\in I$ for some $m\geq 1$ and so $r^m\in \langle r_1,r_2,\cdots,r_n\rangle$ for some $n\geq 1$. Hence $r\in I_n$, that is,  $\sqrt{I}=\bigcup\limits_{n\geq 1}I_n$. Now,  we claim that $s\sqrt{I}\subseteq I_k$. Indeed, let $r\in \sqrt{I}$. Then $r\in I_n$ for some $n\geq 1$. So $sr\in sI_n\subseteq I_k$. Hence $s\sqrt{I}\subseteq I_k$. Consequently, $I$ is radically $S$-finite with respect to $s$.

	$(4)\Rightarrow (3)$ Let $I = \langle r_\alpha \mid \alpha < \kappa \rangle$ be an ideal of $R$, where $\kappa$ is a cardinal. Consider the index set
\[
\mathcal{I} = \left\{ \beta \mid  (s r_\beta)^n \not\in \sqrt{\langle r_\alpha \mid \alpha < \beta \rangle}\  \mbox{for all}\ n\geq 1\} \right  .
\]

We claim that  \textbf{for any $\boldsymbol{r_{\beta'} \in I}$, we have $\boldsymbol{sr_{\beta'} \in \sqrt{\langle r_\beta \mid \beta \in \mathcal{I} \rangle}}$.} Since  the first ordinal is in $\mathcal{I}$, we have $\mathcal{I}$ is nonempty. Let $r_{\beta'} \in I$. If $\beta' \in \mathcal{I}$, we are done. Otherwise, there exists $n_1 \geq 1$ such that
\[
(s r_{\beta_1})^{n_1} \in \langle r_\alpha \mid \alpha < \beta_1 \rangle.
\]
So $(s r_{\beta_1})^{n_1}=\sum k_ir_{\beta_{2,i}}$  with each $\beta_{2,i} < \beta_1$ and each $k_i\in R$. If all such $\beta_{2,i}$s belong to $\mathcal{I}$, we are done. Otherwise, let $\beta_{2,i_1},\cdots \beta_{2,i_{m_2}}$ be the finite set of all $\beta_{2,i}$s which are not in $\mathcal{I}$. Then, for all $k=1,\dots,m_2$,  we have
\[
(s r_{\beta_{2,i_k}})^{n_2} \in \langle r_\alpha \mid \alpha <\beta_2 \rangle
\]
for some $n_2 \geq 1$ and some $\beta_2 < \beta_1$.  Hence $$(s r_{\beta_1})^{n_1n_2m_2}=(s\sum\limits_{\beta_{2,i}\in \mathcal{I}} k_ir_{\beta_{2,i}}+s\sum\limits_{\beta_{2,i}\not\in \mathcal{I}} k_ir_{\beta_{2,i}})^{n_2m_2}=(s\sum\limits_{\beta_{2,i}\not\in \mathcal{I}} k_ir_{\beta_{2,i}})^{n_2m_2}+M,$$
where $M$ is a multiple of $\sum\limits_{\beta_{2,i}\in \mathcal{I}} k_ir_{\beta_{2,i}}$. So $M$ is in $\langle r_\beta \mid \beta \in \mathcal{I} \rangle$. Note that $(s\sum\limits_{\beta_{2,i}\not\in \mathcal{I}} k_ir_{\beta_{2,i}})^{n_2m_2}$ is the linear combination of $(sk_ir_{\beta_{2,i}})^{n_2}$, and so is in $\langle r_\alpha \mid \alpha <\beta_2 \rangle$.
Hence $(s r_{\beta_1})^{n_1n_2m_2}$ is in $\langle r_\beta \mid \beta \in \mathcal{I}\ \mbox{or}\  \beta<\beta_2 \rangle$. Continuing this process, there exists an sequence $r_{\beta_1}, r_{\beta_2},r_{\beta_3},\dots$ with $\beta_1>\beta_2>\beta_3>\dots$. Note that $\{\beta_i\mid i\geq 1\}$ is a subset of the well-ordered set $\kappa$. So $\{\beta_i\mid i\geq 1\}$  has a minimal element, which is a contradiction. Consequently, the claim holds.

If $\mathcal{I}$ is finite, then $I$ is radically $S$-finite with respect to $s$, completing the proof. Now assume that $\mathcal{I}$ is infinite. Consider a countable subset
\[
\mathcal{I}' = \{ \beta_1 < \beta_2 < \beta_3 < \cdots \} \subseteq \mathcal{I}.
\]
Since $\langle r_\beta \mid \beta \in \mathcal{I}' \rangle$ is radically $S$-finite with respect to $s$ by assumption, there exists a finite subset $\mathcal{I}''=\{\beta_1,\beta_2,\dots,\beta_k\} \subseteq \mathcal{I}'$ such that
\[
s \langle r_\beta \mid \beta \in \mathcal{I}' \rangle \subseteq \sqrt{\langle r_\beta \mid \beta \in \mathcal{I}'' \rangle}.
\]
So $(s\beta_{k+1})^{n_{k+1}}\in \langle r_\beta \mid \beta \in \mathcal{I}''\rangle$ for some $n_{k+1}>1.$ Consequently, $\beta_{k+1}\not\in \mathcal{I}$, which is a contradiction.

Consequently, $I$ is radically $S$-finite with respect to $s$, and it follows that $R$ has the uniformly $S$-Noetherian spectrum property with respect to $s$.
\end{proof}

\begin{corollary} \label{count}
	Let $R$ be a commutative ring. Then  $R$ has Noetherian spectrum if and only if each countably generated ideal of $R$ is radically finite.
\end{corollary}
\begin{proof}
	Take $S=\{1\}$ in Theorem \ref{count-us}.
\end{proof}

\section{Rings with uniformly  $S$-$w$-Noetherian spectrum}

In 2014, Kim, Kwon, and  Rhee  \cite{KKR14} defined an integral domain having \emph{strong Mori spectrum}. The related notions can be trivially extend to commutative rings with zero divisors, which can be called rings with $w$-Noetherian spectrum. An ideal $I$ of $R$ is said to be \emph{radically $w$-finite }, if $I = \sqrt{F_w}$ for some finitely generated subideal $F$ of $I$. And a ring $R$ is said to have \emph{$w$-Noetherian spectrum}  if  every ideal of $R$ is radically $w$-finite.

Next, we will  generalize  integral domains with strong Mori spectrum and rings with uniformly $S$-Noetherian spectrum.

\begin{definition} Let $R$ be a ring, and $S$ a multiplicative subset of $R$.
	\begin{enumerate}
		\item An ideal $I$ of $R$ is said to be {radically $S$-$w$-finite }$($abbreviated to be $u$-$S$-$w$-Noetherian spectrum$) ($with respect to $s)$ if $sI \subseteq \sqrt{F_w}$ for some finitely generated subideal $F$ of $I$ and some $s \in S$.
		\item  A ring $R$ is said to have {uniformly  $S$-$w$-Noetherian spectrum} $($abbreviated to be $u$-$S$-$w$-Noetherian spectrum$) ($with respect to $s)$ if there exists $s\in S$ such that every ideal of $R$ is radically $S$-$w$-finite  with respect to $s$.
	\end{enumerate}
\end{definition}
If $R$ is a $\DW$-ring, then rings with uniformly  $S$-$w$-Noetherian spectrum have uniformly  $S$-Noetherian spectrum. It is clear that if $S_1$ is a multiplicative subset of $R$ contained in $S_2$, then rings with uniformly  $S_1$-$w$-Noetherian spectrum also have uniformly  $S_2$-$w$-Noetherian spectrum. The saturation $S^{\ast}$ of a multiplicative subset $S$ of $R$ is defined as
\[
S^{\ast} = \{s \in R \mid s_1 = ss_2 \text{ for some } s_1, s_2 \in S\}.
\]
A multiplicative subset $S$ of $R$ is called saturated if $S = S^{\ast}$. Note that $S^{\ast}$ is always a saturated multiplicative subset containing $S$. Clearly, rings with uniformly  $S^{\ast}$-$w$-Noetherian spectrum are exactly those with uniformly $S$-$w$-Noetherian spectrum.

We trivially have the following implications:
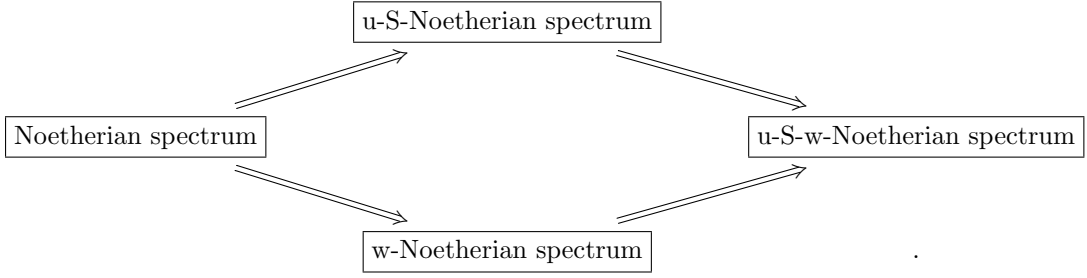
\[
\xymatrix@R=20pt@C=25pt{
&\boxed{$u$\text{-}$S$\text{-Noetherian spectrum}} \ar@{=>}[rd]&\\
{\boxed{\text{Noetherian spectrum}}}\ar@{=>}[ru]	\ar@{=>}[rd]& &{\boxed{$u$\text{-}$S$\text{-}$w$\text{-Noetherian spectrum}}}\\
	&\boxed{$w$\text{-Noetherian spectrum}} \ar@{=>}[ru]&.
}
\]

Rings with uniformly  $S$-$w$-Noetherian spectrum may not have uniformly  $S$-Noetherian spectrum.
\begin{example}  Let $R=k[x_1,x_2,\cdots]$ with $k$ a field. Set $S=\{1\}.$ Then  $R$ has uniformly  $S$-$w$-Noetherian spectrum as  $R$ is a UFD. However,  $R$ does not have uniformly  $S$-Noetherian spectrum, since the prime ideals $\langle x_1\rangle\subseteq\langle x_1,x_2\rangle\subseteq\cdots$ is not stationary.

\end{example}

Let $R$ be a ring and $S$ a multiplicative subset of $R$. For any $s\in S$, there is a  multiplicative subset $S_s=\{1,s,s^2,....\}$ of $S$. We denote by $M_s$ the localization of $M$ at $S_s$. Certainly, $M_s\cong M\otimes_RR_s$. A multiplicative subset of $R$ is said to be \emph{regular} if it consists of non-zero-divisors.

\begin{proposition}\label{local}
	Let $R$ be a ring, $S$ a regular multiplicative subset of $R$ and $s\in S$. Suppose $R$ has  uniformly $S$-$w$-Noetherian spectrum with respect to $s$. Then $R_s$ has $w$-Noetherian spectrum.
\end{proposition}
\begin{proof} Let $I_s$ be an ideal of $R_s$ with $I$ an ideal of $R$. Then there is a finitely generated ideal $F$ of $I$ such that if $sI \subseteq \sqrt{F_w}=(\sqrt{F})_w$ (Note that $\sqrt{I_w}=(\sqrt{I})_w$ for any ideal $I$ of $R$). For any $r\in I$, we have $sr\in \sqrt{F}_w$ Hence there is an ideal $J\in\GV(R)$ such that $Jsr\in \sqrt{F}$. Hence $J_s\frac{r}{1}\in \sqrt{F}_s=\sqrt{F_s}$. Since $J_s\in\GV(R_s)$ by \cite[Lemma 4.3(1)]{WQ15}, $\frac{r}{1}\in\sqrt{(F_s)_w}.$
Hence 	$I_s\subseteq \sqrt{(F_s)_w}$. Consequently, $R_s$ has $w$-Noetherian spectrum.
\end{proof}

\begin{proposition}\label{mutl-usm}
Let $n \geq 2$ be an integer, $R_{1},\dots ,R_{n}$ be  rings and $S_{1},\dots ,S_{n}$ be multiplicative subsets of $R_{1},\dots ,R_{n}$, respectively. Then
the following statements are equivalent.
\begin{enumerate}
	\item $R_{k}$ has $S_{k}$-$w$-Noetherian spectrum for all $k = 1,\dots ,n$.
	\item $\prod_{k=1}^{n} R_{k}$ has $\Bigl(\prod_{k=1}^{n} S_{k}\Bigr)$-$w$-Noetherian spectrum.
\end{enumerate}
\end{proposition}
\begin{proof}
Write $R := \prod_{k=1}^{n} R_{k}$ and $S := \prod_{k=1}^{n} S_{k}$.
Let  $w$ be the $w$-operation of $R$, and $w_k$ be the $w$-operation of $R_k$ for each $k = 1,\dots ,n$.
	
	$(1) \Rightarrow (2)$
	Let $I$ be an ideal of $R$. Then for each $k = 1,\dots ,n$, we can find an
	ideal $I_{k}$ of $R_{k}$ such that $I = \prod_{k=1}^{n} I_{k}$. Since each $R_{k}$ has $S_{k}$-$w_k$-Noetherian spectrum,
	there exist an element $s_{k}\in S_{k}$ and a finitely generated subideal $J_{k}$ of $I_{k}$ such that
	\[
	s_{k}I_{k}\subseteq\sqrt{(J_{k})_{w_k}}.
	\]
	Note that $\prod_{k=1}^{n}\sqrt{(J_{k})_{w_k}}=\sqrt{(\prod_{k=1}^{n} J_{k})_w}$; so
	\[
	(s_{1},\dots ,s_{n})I\subseteq\sqrt{(\prod_{k=1}^{n} J_{k})_w}.
	\]
	Hence $I$ is radically $S$-$w$-finite. Thus $R$ has $S$-$w$-Noetherian spectrum.
	
	(2) $\Rightarrow$ (1)
	Suppose that $R$ has $S$-$w$-Noetherian spectrum and for each $k\in\{1,\dots ,n\}$, let $I_{k}$ be an ideal of $R_{k}$. Then $I := \prod_{k=1}^{n} I_{k}$ is an ideal of $R$; so we
	can find an element $(s_{1},\dots ,s_{n})\in S$ and a finitely generated subideal $J$ of $I$ such
	that $(s_{1},\dots ,s_{n})I\subseteq\sqrt{J_w}$. Note that for each $k = 1,\dots ,n$, there exists a finitely
	generated ideal $J_{k}$ of $R_{k}$ such that $J = \prod_{k=1}^{n} J_{k}$. Hence for each $k\in\{1,\dots ,n\}$,
	\[
	s_{k}I_{k}\subseteq\sqrt{(J_{k})_{w_k}}.
	\]
	Thus for each $k = 1,\dots ,n$, $I_{k}$ is radically $S_{k}$-$w_k$-finite, which means that
	$R_{k}$ has $S_{k}$-$w_k$-Noetherian spectrum.
\end{proof}

Rings with uniformly  $S$-$w$-Noetherian spectrum may not have uniformly  $w$-Noetherian spectrum.
\begin{example} Let $R_1$ be a ring with $w$-Noetherian spectrum, and $R_2$ be a ring without $w$-Noetherian spectrum. Set $R=R_1\times R_2$, and $S=\{1\}\times\{0,1\}$. Then $R$ has uniformly $S$-$w$-Noetherian spectrum  but does not have  $w$-Noetherian spectrum by Proposition \ref{mutl-usm}.
\end{example}

Let $P$ be a prime ideal of $R$. We call a ring $R$ has uniformly $P$-$w$-Noetherian spectrum provided that $R$ has uniformly $(R-P)$-$w$-Noetherian spectrum.

\begin{proposition}
The following statements are equivalent for a  ring $R$.
\begin{enumerate}
	\item $R$ has $w$-Noetherian spectrum.
		\item $R$ has uniformly $P$-$w$-Noetherian spectrum for any $P \in \mathrm{Spec}(R)$.
	\item $R$ has uniformly $M$-$w$-Noetherian spectrum for any $M \in \mathrm{Max}(R)$.
	\item $R$ has uniformly $P$-$w$-Noetherian spectrum for any $P \in w\mbox{-}\mathrm{Spec}(R)$.
	\item $R$ has uniformly $M$-$w$-Noetherian spectrum for any $M \in w\mbox{-}\mathrm{Max}(R)$.
\end{enumerate}
\end{proposition}
\begin{proof}  $(1)\Rightarrow (2) \Rightarrow (3) $ and $(1)\Rightarrow (4) \Rightarrow (5) $: These implications are trivial.

$(3) \Rightarrow (1)$: Let $I$ be an ideal of $R$. By (3), for any $M \in \mathrm{Max}(R)$, there exists an $s_M \in R \setminus M$ such
that $s_M I \subseteq \sqrt{(F_M)_w}$ for some finitely generated subideal $F_M$ of
$I$. Let $K$ be the ideal of $R$ generated by the set $\{s_M \mid M \in \mathrm{Max}(R)\}$.  Then $K = R$. Hence
$1 = s_{M_1}r_1 + \dots + s_{M_n}r_n$ for some $r_1,\dots,r_n \in R$. Set $F := \sum_{i=1}^n F_{M_i}$. Then $F$ is
a finitely generated subideal of $I$. Furthermore, $I = 1 \cdot I = (s_{M_1}r_1 + \dots + s_{M_n}r_n)I \subseteq
s_{M_1}I + \dots + s_{M_n}I \subseteq \sum_{i=1}^n \sqrt{(F_{M_i})_w} \subseteq \sqrt{\sum_{i=1}^n (F_{M_i})_w} \subseteq \sqrt{F_w} \subseteq \sqrt{I_w}$. Hence $R$ has
$w$-Noetherian spectrum.

$(5) \Rightarrow (1)$: Let $I$ be an ideal of $R$. By (5), for any $M \in  w\mbox{-}\mathrm{Max}(R)$, there exists an $s_M \in R \setminus M$ such
that $s_M I \subseteq \sqrt{(F_M)_w}$ for some finitely generated subideal $F_M$ of
$I$. Let $K$ be the ideal of $R$ generated by the set $\{s_M \mid M \in  w\mbox{-}\mathrm{Max}(R)\}$.  Then $K_w= R$. Hence there exists $\{s_{M_1},\dots, s_{M_n}\}$ such that $J:=\langle s_{M_1},\dots, s_{M_n}\rangle\in \GV(R)$. Set $F := \sum_{i=1}^n F_{M_i}$. Then $F$ is
a finitely generated subideal of $I$. Furthermore, $JI =
s_{M_1}I + \dots + s_{M_n}I \subseteq \sum_{i=1}^n \sqrt{(F_{M_i})_w} \subseteq \sqrt{\sum_{i=1}^n (F_{M_i})_w} \subseteq \sqrt{F_w} \subseteq \sqrt{I_w}$. Hence $I_w\subseteq \sqrt{F_w} \subseteq \sqrt{I_w}$
, implying that
$R$ has
$w$-Noetherian spectrum.
\end{proof}

\begin{lemma}\label{w-iso}
	Let $R$ be a ring, $S$ a multiplicative subset of $R$ and $s\in S$.	Let $I$ and $K$ be two ideals of $R$ such that $I_w=K_w$. Then $I$ is radically $S$-$w$-finite with respect to $s$ if and only if so is $K$.
\end{lemma}
\begin{proof}  Suppose $I$ is radically $S$-$w$-finite with respect to $s$. $sI \subseteq \sqrt{F_w}$ for some finitely generated subideal $F=\langle f_1,\cdots,f_n\rangle$ of $I$.  Since $I_w= K_w$, there is $J\in\GV(R)$ such that $JF\subseteq K.$ Hence  $sK \subseteq sK_w=sI_w \subseteq (sI)_w\subseteq (\sqrt{F_w})_w=\sqrt{F_w}=\sqrt{(J_wF_w)_w}=\sqrt{(JF)_w},$  by \cite[[Proposition 6.2.3(5), Proposition 6.2.3]{FK24}. Since $JF$ is finitely generated,   $K$ is radically $S$-$w$-finite with respect to $s$.  The converse is similarly true.
\end{proof}

\begin{theorem}\label{main-1}
	Let $R$ be a ring, $S$ a multiplicative subset of $R$ and $s\in S$.
	Then the following statements are equivalent.
	\begin{enumerate}
		\item $R$ has uniformly $S$-$w$-Noetherian spectrum with respect to $s$.
		\item Each ascending chain of radical \mbox{$w$-ideal}s of $R$ is stationary   with respect to $s\in S$.
		\item Each radical \mbox{$(w$-$)$ideal} of $R$ is radically $S$-$w$-finite with respect to $s$.
		\item Each prime \mbox{$(w$-$)$ideal} of $R$ is radically $S$-$w$-finite with respect to $s$.
		\item Each  countably generated ideal  of $R$ is radically $S$-$w$-finite with respect to $s$.
	\end{enumerate}
\end{theorem}
\begin{proof} The equivalence of $w$-ideals and  ideals in $(3),(4)$ can be deduced by Lemma \ref{w-iso}.
	
	$(1)\Rightarrow (3)\Rightarrow (4)$ Trivial.
	
 $(4)\Rightarrow (3)$ On contrary, assume that there exists a radical $w$-ideal $K$  of  $R$ which is not  radically $S$-$w$-finite with respect to $s$.   Set
 \begin{center}
 	$\mathcal{F}=\{$radical $w$-ideals of $R$  that are not radically $S$-$w$-finite type with respect to $s\}.$
 \end{center}
Since $K \in \mathcal{F}$,
 we have  $\mathcal{F} \neq \emptyset$.  Ordered $\mathcal{F}$ by inclusion.
Let $\{I_\lambda\}_{\lambda \in \Lambda}$  be a chain in $\mathcal{F}$  and set $I = \bigcup_{\lambda \in \Lambda} I_\lambda.$
Then $I $ is a radical $w$-ideal of $R.$
Indeed, let $r \in \sqrt{I}$, then $r^k \in I = \bigcup_{\lambda \in \Lambda} I_\lambda $ for some positive integer $k.$
So $r^k \in I_\lambda$ for some $\lambda \in \Lambda.$
 Thus $r \in \sqrt{I_\lambda} = I_\lambda \subseteq I.$ And note that the directed unions of $w$-ideals is also a $w$-ideal (see \cite[Proposition 6.1.19]{FK24}). We claim that $I \in \mathcal{F}$.
Assume that $I \notin \mathcal{F}.$
There exists a finitely generated subideal $J$ of $I$  such that
 $sI \subseteq \sqrt{J_w}$. Since $J$ is finitely generated, there exists $\lambda \in \Lambda$ such that $J \subseteq I_\lambda$, and so  $sI \subseteq sI_\lambda \subseteq \sqrt{J_w} \subseteq \sqrt{(I_\lambda)_w}$,  a contradiction.
Hence, by Zorn's Lemma, there is a maximal element $P$ of $\mathcal{F}$.

We claim that the maximal element $P$ of $\mathcal{F}$ is a prime ideal of  $R.$ Indeed, on contrary, assume that $P$ is not a prime ideal of $R.$
 There exist $a, b \in R \setminus P$ such that $ab \in P.$
We put $I := P + aR$ and $K := P + bR$.
Since  $IK = P^2 + (aR)P + (bR)P + (ab)R \subseteq P,$ we have $\sqrt{IK} \subseteq \sqrt{P} = P.$
 Also $P \subseteq I \cap K \subseteq \sqrt{I \cap K} = \sqrt{IK},$ thus $P = \sqrt{IK}$.
On the other hand, $P \subsetneq I \subseteq \sqrt{I}$ and $P \subsetneq K \subseteq \sqrt{K}.$
Then by maximality of $P,$ there exist finitely generated subideals $I'$  of  $I$ and $K'$ of $K$ such that  $s\sqrt{I} \subseteq \sqrt{(I')_w}$  and  $s\sqrt{K} \subseteq \sqrt{(K')_w}$.
Let $x \in \sqrt{I}$ and $y \in \sqrt{K}$. Then $sx \in \sqrt{(I')_w}$ and $sy \in \sqrt{(K')_w}$,  which implies that  $(sx)^k \in (I')_w$ and $(sy)^{k'} \in (K')_w$ for some positive integers $ k, k'$.
Thus there exists $L_1,L_2\in\GV(R)$ such that  $(L_1L_2sxy)^{k+k'} \subseteq  I'K'$,  so $ L_1L_2 sxy \in \sqrt{I'K'}$. Hence $ sxy \in (\sqrt{I'K'})_w=\sqrt{(I'K')_w}$
 This shows that $s\sqrt{I}\sqrt{K} \subseteq \sqrt{(I'K')_w}$.
Thus $sP \subseteq s\sqrt{IK} = s\sqrt{I}\sqrt{K} \subseteq \sqrt{(I'K')_w},$  which  contradicts the maximality of $P$.
So $P$ is a prime ideal of $R$  which is not radically $S$-$w$-finite with respect to  $s$,
a contradiction.

 $(3)\Rightarrow (1)$  Let  $I$ be an ideal of $R$. By (3), $\sqrt{I}$ is radically $S$-$w$-finite with respect to  $s$, then $s\sqrt{I}\subseteq  \sqrt{(a_1,\dots,a_n)_w}$ for some $a_1,\dots,a_n \in \sqrt{I}$.
 So for all $1 \le i \le n$, $a_i^{r_i} \in I$ for some positive integer $r_i$. For each $1 \le i \le n$, let
 $b_i = a_i^{r_i}$ and let $F := (b_1,\dots,b_n)$. Then $F$ is a finitely generated subideal of $I$.
 Moreover, $sI \subseteq s \sqrt{I} \subseteq \sqrt{(a_1,\dots,a_n)_w} \subseteq \sqrt{F_w} \subset \sqrt{I_w}$.


	 $(1)\Rightarrow (2)$
Let $I_1\subseteq I_2\subseteq I_3\subseteq\cdots$ be ascending chain of radical $w$-ideals of $R$. Let $I=\bigcup\limits_{i=1}^{\infty} I_i$. Then $I$ is radically $S$-$w$-finite with respect to $s$, that is, $sI \subseteq \sqrt{J_w}$ for some finitely generated subideal $J$ of $I$. Assume that $J\subseteq \bigcup\limits_{i=1}^{k} I_i$. Then $$sI \subseteq \sqrt{(\bigcup\limits_{i=1}^{k} I_i)_w}=\sqrt{\bigcup\limits_{i=1}^{k} I_i}=\bigcup\limits_{i=1}^{k}\sqrt{I_i}=\bigcup\limits_{i=1}^{k}I_i=I_k.$$
Consequently, $sI_n\subseteq I_k$ for any $n\geq k.$

$(2)\Rightarrow (5)$ Let $I=\langle r_1,r_2,\cdots\rangle$ be a countably generated ideal of $R$. Set  $I_n=\sqrt{\langle r_1,r_2,\cdots,r_n\rangle}$. Then $\{(I_n)_w\mid n\geq 1\}$ is an ascending chain of radical $w$-ideals of $R$ (Note that $(\sqrt{I})_w = \sqrt{I_w}$ for each ideal $I$ of $R$). So there is a positive integer $k$ such that $s(I_n)_w\subseteq (I_k)_w$ for all $n\geq k$. Note that $\sqrt{I_w}=\sqrt{I}_w=\bigcup\limits_{n\geq 1}(I_n)_w$. Now,  we claim that $s\sqrt{I_w}\subseteq (I_k)_w$. Indeed, let $r\in \sqrt{I_w}$. Then $r\in (I_n)_w$ for some $n\geq 1$. So $sr\in s(I_n)_w\subseteq (I_k)_w$. Hence $sI\subseteq sI_w\subseteq s\sqrt{I_w}\subseteq (I_k)_w$. Consequently, $I_w$ is radically $S$-$w$-finite with respect to $s$.

$(5)\Rightarrow (1)$ Let $I = \langle r_\alpha \mid \alpha < \kappa \rangle$ be an ideal of $R$, where $\kappa$ is a cardinal. Consider the index set
\[
\mathcal{I} = \left\{ \beta \mid  J(s r_\beta)^n \not\subseteq \sqrt{\langle r_\alpha \mid \alpha < \beta \rangle}\  \mbox{for all}\ n\geq 1\ \mbox{and all}\ J\in \GV(R)\} \right  .
\]

We claim that  \textbf{for any $\boldsymbol{r_{\beta'} \in I}$, we have $\boldsymbol{Jsr_{\beta'} \in \sqrt{\langle r_\beta \mid \beta \in \mathcal{I} \rangle}}$ for some $\boldsymbol{J\in \GV(R)}$.} Since  the first ordinal is in $\mathcal{I}$, we have $\mathcal{I}$ is nonempty. Let $r_{\beta'} \in I$. If $\beta' \in \mathcal{I}$, we are done. Otherwise, there exists $n_1 \geq 1$ and $J_1=\langle j_{1,1},\dots, j_{1,l_1}\rangle\in \GV(R)$ such that
\[
(J_1s r_{\beta_1})^{n_1} \subseteq \langle r_\alpha \mid \alpha < \beta_1 \rangle.
\]
So $(j_{1,l}s r_{\beta_1})^{n_1}=\sum k_ir_{\beta_{2,i}}$  with each $\beta_{2,i} < \beta_1$, each $k_i\in R$ and each $l=1,\dots,l_1$. If all such $\beta_{2,i}$s belong to $\mathcal{I}$, we are done. Otherwise, let $\beta_{2,i_1},\cdots \beta_{2,i_{m_2}}$ be the finite set of all $\beta_{2,i}$s which are not in $\mathcal{I}$. Then, for all $k=1,\dots,m_2$, there exists  $J_2=\langle j_{2,1},\dots, j_{2,l_2}\rangle\in \GV(R)$ such that
\[
(J_2s r_{\beta_{2,i_k}})^{n_2} \subseteq \langle r_\alpha \mid \alpha <\beta_2 \rangle
\]
for some $n_2 \geq 1$ and some $\beta_2 < \beta_1$.  Hence for all $l=1,\dots,l_1$ and $l'=1,\dots,l_2$, we have \begin{align*}(j_{1,l}j_{2,l'}s r_{\beta_1})^{n_1n_2m_2}&=(j_{1,l}j_{2,l'}s\sum\limits_{\beta_{2,i}\in \mathcal{I}} k_ir_{\beta_{2,i}}+j_{1,l}j_{2,l'}s\sum\limits_{\beta_{2,i}\not\in \mathcal{I}} k_ir_{\beta_{2,i}})^{n_2m_2}\\
	&=(j_{1,l}j_{2,l'}s\sum\limits_{\beta_{2,i}\not\in \mathcal{I}} k_ir_{\beta_{2,i}})^{n_2m_2}+M,\end{align*}
where $M$ is a multiple of $\sum\limits_{\beta_{2,i}\in \mathcal{I}} k_ir_{\beta_{2,i}}$. So $M$ is in $\langle r_\beta \mid \beta \in \mathcal{I} \rangle$. Note that $(j_{1,l}j_{2,l'}s\sum\limits_{\beta_{2,i}\not\in \mathcal{I}} k_ir_{\beta_{2,i}})^{n_2m_2}$ is the linear combination of $(j_{1,l}j_{2,l'}sk_ir_{\beta_{2,i}})^{n_2}$, and so is in $\langle r_\alpha \mid \alpha <\beta_2 \rangle$.
Hence $(J_1J_2s r_{\beta_1})^{n_1n_2m_2}$ is contained in $\langle r_\beta \mid \beta \in \mathcal{I}\ \mbox{or}\  \beta<\beta_2 \rangle$. Continuing this process, there exists an sequence $r_{\beta_1}, r_{\beta_2},r_{\beta_3},\dots$ with $\beta_1>\beta_2>\beta_3>\dots$. Note that $\{\beta_i\mid i\geq 1\}$ is a subset of the well-ordered set $\kappa$. So $\{\beta_i\mid i\geq 1\}$  has a minimal element, which is a contradiction. Consequently, the claim holds.

If $\mathcal{I}$ is finite, then $I$ is radically $S$-$w$-finite with respect to $s$, completing the proof. Now assume that $\mathcal{I}$ is infinite. Consider a countable subset
\[
\mathcal{I}' = \{ \beta_1 < \beta_2 < \beta_3 < \cdots \} \subseteq \mathcal{I}.
\]
Since $\langle r_\beta \mid \beta \in \mathcal{I}' \rangle$ is radically $S$-$w$-finite with respect to $s$ by assumption, there exists a finite subset $\mathcal{I}''=\{\beta_1,\beta_2,\dots,\beta_k\} \subseteq \mathcal{I}'$ such that
\[
s \langle r_\beta \mid \beta \in \mathcal{I}' \rangle \subseteq \sqrt{\langle r_\beta \mid \beta \in \mathcal{I}'' \rangle_w}.
\]
So $(Js\beta_{k+1})^{n_{k+1}}\subseteq \langle r_\beta \mid \beta \in \mathcal{I}''\rangle$ for some $n_{k+1}>1$ and some  $J\in\GV(R)$. Consequently, $\beta_{k+1}\not\in \mathcal{I}$, which is a contradiction.

Consequently, $I$ is radically $S$-$w$-finite with respect to $s$, and it follows that $R$ has the uniformly $S$-$w$-Noetherian spectrum property with respect to $s$.
\end{proof}

\begin{corollary}\label{char-w-noe-spect}
	Let $R$ be a ring. Then the following statements are equivalent.
\begin{enumerate}
	\item $R$ has $w$-Noetherian spectrum.
	\item Each ascending chain of radical $w$-ideals of $R$ is stationary.
	\item Each radical \mbox{$(w$-$)$ideal} of $R$ is radically $w$-finite.
	\item Each prime \mbox{$(w$-$)$ideal} of $R$ is radically $w$-finite.
	\item Each  countably generated ideal  of $R$ is radically $w$-finite.
\end{enumerate}
\end{corollary}
\begin{proof}
	Take  $S=\{1\}$ in Theorem \ref{main-1}.
\end{proof}

To build a connection between the $w$-operation case and ordinal case,
Wang and Kim {\cite{WK15}} introduced the notion of the \emph{$w$-Nagata ring}, $R\{x\}$, of $R$. It is a localization of $R[X]$ at the multiplicative closed set
$$S_w=\{f\in R[x]\,|\, c(f)\in \GV(R)\},$$
\noindent where $c(f)$ is the content of $f$. Note that $R\{x\}$ can naturally be seen as an extension of $R$.  The \emph{$w$-Nagata module} $M\{x\}$ of an $R$-module $M$ is defined as $M\{x\}=M[x]_{S_w}\cong M\bigotimes_R R\{x\}$.

\begin{lemma}\label{rad-w-Nag} Let $I$ be an ideal of $R$. Then 
$\sqrt{I}\{X\}=\sqrt{I\{X\}}.$
\end{lemma}
\begin{proof} Let $\frac{f}{g}\in \sqrt{I}\{X\}$ with $f=a_nX^n+\cdots+a_0\in \sqrt{I}[X]$ and $c(g)\in \GV(R)$. Then there exists $m\geq 1$ such that $a_i^m\in I$ for each $i=0,\cdots,n$. So $f^{m(n+1)}\in I[x]$. Hence $(\frac{f}{g})^{m(n+1)}\in  I\{X\}.$ Consequently,  $\frac{f}{g}\in \sqrt{I\{X\}}$. Hence$\sqrt{I}\{X\}\subseteq \sqrt{I\{X\}}$.
	
	On the other hand, then $a_n^m\in$ obviously. Let $\frac{f}{g}\in \sqrt{I\{X\}}$. Then there exists $m\geq 1$ such that $(\frac{f}{g})^m\in I\{X\}$ with $f^m\in I[X]$ and $c(g^m)\in \GV(R)$. Since $c(g^m)\subseteq c(g)$, we have $c(g)\in\GV(R).$ Assume $f=a_nX^n+\cdots+a_0.$ We will show each $a_i\in\sqrt{I}$ by induction on $n$. If $n=0$, it is trivial. For general case, let $f_1=f-a_nX^n$. Then $f^m=\sum\limits_{i=0}^mC_m^ia_n^{m-i}X^{n(m-i)}f_1^i=a_n^mX^{nm}+\cdots\in I[X].$ Then $a_n^m\in I$, and thus  $a_n\in\sqrt{I}$.
	Note that $f_1^{2m}\in I[x]$ and $\deg(f_1)<\deg(f)$. One can show that each $a_i\in\sqrt{I}$ by induction. So $f\in\sqrt{I}[x]$.  Hence$\sqrt{I}\{X\}\supseteq \sqrt{I\{X\}}$.
	
\end{proof}

\begin{proposition} \label{swf-nag}
	Let $R$ be a ring, $S$ a multiplicative subset of $R$ and $s\in S$.	 Then
	an ideal $I$ of $R$ is radically $S$-$w$-finite with respect to $s$ if and only if $I\{X\}$ is  radically $S$-finite  with respect to $s$. 	
\end{proposition}
\begin{proof}
	Suppose $I$ is radically $S$-$w$-finite with respect to $s$. Then there exists a finitely  generated subideal $F$ of $I$ such that $sI\subseteq \sqrt{F_w}$. Then  $sI\{X\}\subseteq \sqrt{F_w}\{X\}=\sqrt{F_w\{X\}}=\sqrt{F\{X\}}$ by Lemma \ref{rad-w-Nag}.  Hence  $I\{X\}$ is radically $S$-$w$-finite  with respect to $s$. 	
	
On the other hand, suppose $I\{X\}$ is radically $S$-finite  with respect to $s$. Then $sI\{X\}\subseteq\sqrt{A_{S_w}}$ for some finitely generated subideal  $A$  of $I[X]$. So $c(A)$ is a finitely generated subideal of $I$. Since $A\subseteq c(A)[\textsc{X}]\subseteq I[X]$,
	we have $sI\{X\} \subseteq\sqrt{A_{S_w}}\subseteq \sqrt{c(A)\{X\}}=\sqrt{c(A)}\{X\}$ by Lemma \ref{rad-w-Nag}.  Consequently, we have $sI\subseteq(sI)_w\subseteq (\sqrt{c(A)}))_w=\sqrt{c(A)_w}$. Hence, $I$ is radically $S$-$w$-finite  with respect to $s$.	
\end{proof}

\begin{theorem}\label{main-2}
	Let $R$ be a ring, $S$ a  multiplicative subset of $R$ and $s\in S$.
	Then  the following statements are equivalent.
	\begin{enumerate}
		\item $R$ has uniformly $S$-$w$-Noetherian spectrum.
		\item  $R[X]$ has  uniformly $S$-$w$-Noetherian spectrum.
		\item
		$R\{X\}$ has  uniformly $S$-Noetherian spectrum.
	\end{enumerate}
\end{theorem}

\begin{proof} We denote by the $w$-operation of $R$ by $w$, and the $w$-operation of $R[x]$ by $W$, respectively.
	
$(1)\Rightarrow (2)$	Suppose that $R$ has the uniformly $S$-$w$-Noetherian spectrum property with respect to some $s \in S$. We
	claim that $R[X]$ also has the uniformly $S$-$w$-Noetherian spectrum property with respect to $s$.
	
	Let $P$ be a $w$-ideal of $R[X]$ that is maximal among those that are not radically $S$-$w$-finite with
	respect to $s$. Then $P$ is a prime ideal of $R[X]$. Indeed, suppose $xy \in P$ with $x \notin P$ and $y \notin P$.
	Then both $P + Rx$ and $P + Ry$ are $S$-$w$-finite with respect to $s$, implying
	\[
	s(P + Rx) \subseteq \sqrt{(I_x)_W},\qquad
	s(P + Ry) \subseteq \sqrt{(I_y)_W}
	\]
	for some finitely generated ideals $I_x$ and $I_y$ of $P + Rx$ and $P + Ry$, respectively. It follows that
\begin{align*}
	s^2(P + Rx)(P + Ry) &\subseteq \sqrt{(I_x)_W}\,\sqrt{(I_y)_W} \\
	&= \sqrt{(I_x)_W(I_y)_W}\\
		&\subseteq \sqrt{((I_x)_W(I_y)_W)_W}\\
	&\subseteq \sqrt{(I_xI_y)_W}.
\end{align*}
	Thus, $s^2P \subseteq \sqrt{(I_xI_y)_W}$, leading to a contradiction since $sP \subseteq \sqrt{(I_xI_y)_W}$.
	
	Set $\mathfrak p = P \cap R$. Then $\mathfrak p$ is a prime $w$-ideal of $R$. So $s\mathfrak p\subseteq \sqrt{K_w}$ for some finitely generated subideal $K$  of $\mathfrak p$. Set $M=(\mathfrak pR[x])_W$.
Then $sM=s(\mathfrak pR[x])_W\subseteq (\sqrt{K}R[x])_W$ by \cite[Proposition 6.2.3]{FK24}. So $M$ is a  radically $S$-$w$-finite $w$-ideal  with respect to $s$. Note $M=(\mathfrak pR[x])_W=((P \cap R)R[x])_W\subseteq P_W=P.$ Hence $M\subsetneq P.$

Let $f$ be an element of least degree in $P \setminus M$. Let $n$ be the degree of $f$ and let $a$ be the coefficient of $X^n$ in $f$. Then $n \geq 1$ and $a \notin P$. Indeed, if $n = 0$, then $f = a \in P \cap R \subseteq M$. This contradicts the choice of $f$. Hence $n \geq 1$. Next, suppose to the contrary that $a \in P$. Then $f - aX^n \in P$. Since $\deg(f - aX^n) < \deg(f)$, $f - aX^n \in M$ by the minimality of $n$. Also, note that $a \in P \cap R$; so $aX^n \in (P \cap R)R[X] \subseteq M$. Therefore $f = (f - aX^n) + aX^n \in M$, which is a contradiction. Hence $a \notin P$.

Since $a \notin P$ and $P$ is a $w$-ideal of $R[x]$, $P \subsetneq (P + (a))_W$. So by the maximality of $P$, $(P + (a))_W$ is a $w$-finite $w$-ideal of $R[X]$. Therefore there exists a finitely generated ideal $F$ of $R[X]$ such that $s(P + (a))_W\subseteq \sqrt{F_W}$, where $F = B + (a)$ for some finitely generated subideal $B$ of $P$, that is,
\[
s(P + (a))_W \subseteq \sqrt{(B + (a))_W}.
\]

Let $g \in P$. If $g \in M$, then $ag \in ((f) + M)_W$. Now suppose that  $g \in P \setminus M$. Let $m$ be the degree of $g$. Then $m \geq n$. Notice the minimality of $n$. By a routine iterative calculation, one can find suitable elements $q, r \in R[X]$ such that
\[
a^m g = f q + r
\quad\text{and}\quad
\deg(r) < \deg(f).
\]
Therefore $r = a^m g - f q \in P$. By the minimality of $n$, we have $r \in M$. Also, note that
\[
(ag)^m = f q g^{m-1} + r g^{m-1} \in ((f) + M)_W,
\]
so $ag \in \sqrt{((f) + M)_W}$. Hence $aP \subseteq \sqrt{((f) + M)_W}$.

Now, we obtain
\begin{align*}
s^2P^2 &\subseteq s^2P(P + (a))_W \\
&\subseteq s^2P\sqrt{(B + (a))_W} \\
&\subseteq s(\sqrt{P(B + (a))})_W \\
&\subseteq s(\sqrt{B + aP})_W \\
&\subseteq s\sqrt{(B + ((f) + M)_W)_W}\\
&\subseteq\sqrt{\bigl(B + (f) + sM\bigr)_W}\\
&\subseteq\sqrt{\bigl(B + (f) + KR[X]\bigr)_W}.
\end{align*}
Hence $sP\subseteq
\sqrt{\bigl(B + (f) + KR[X]\bigr)_W}$
implying  $P$ is a $S$-$w$-finite ideal of $R[X]$ with respect to $s$, as $B + (f) + KR[X]$ is a finitely generated subideal of $P$. This is a contradiction to the fact that $P$ is not an $S$-$w$-finite ideal of $R[X]$ with respect to $s$. Thus $R[X]$ has $S$-$w$-Noetherian spectrum with respect to $s$.
	
	$(2) \Rightarrow (1)$ Let $I$ be an ideal of $R$. Then $IR[X]$ is an ideal of $R[X]$. Since $R[X]$ has uniformly $S$-$w$-Noetherian spectrum with respect to $s$, $IR[X]$ is a radically $S$-$w$-finite ideal of $R[X]$ with respect to $s$. So there is a finitely generated subideal $K$ of $I$ such that
	\[
	sIR[X]\subseteq \sqrt{(KR[X])_W}.
	\]
	Therefore, it follows by  follows by \cite[Proposition 6.5.2]{FK24} that
	\[
sI_w=sI_w R[X] \cap R
	= s(IR[X])_W \cap R
\subseteq \sqrt{(KR[X])_W }\cap R
	= \sqrt{K_w R[X]\cap R}
	=\sqrt{K_w}.
	\]
Hence $I$ is a $S$-$w$-finite ideal of $R$ with respect to $s$. Thus $R$ has $S$-$w$-Noetherian spectrum with respect to $s$.

	$(2)\Rightarrow (3)$:  Let $I$ be an ideal of $R\{x\}.$ Then there exists an ideal $K$ of $R[X]$ such
that $I = K_{S_w}$. Since $R[X]$ has $w$-Noetherian spectrum, there exists a finitely generated
subideal $L$ of $K$ such that $sK\subseteq \sqrt{L_W}$. So for any $f\in K$ there is $J\in\GV(R)$ and a positive integer $n\geq 1$ such that $(Jsf)^n\subseteq L$. Since $(J^n)_{S_w}=R\{x\}$, we have $(s\frac{f}{1})^n\in L_{S_w}$. Hence $sI\subseteq \sqrt{L_{S_w}}$. Consequently, 	$R\{x\}$ has uniformly $S$-Noetherian spectrum with respect to $s$.

$(3)\Rightarrow (1)$: Suppose $R\{x\}$ has uniformly $S$-Noetherian spectrum with respect to  $s$. Let $I$ be an ideal of $R$. Then $I\{x\}$ is radically $S$-finite  with respect to $s$ as an $R\{x\}$-ideal. So $I$ is radically $S$-$w$-finite with respect to $s$	by Proposition \ref{swf-nag}. Hence, $R$ is has uniformly $S$-$w$-Noetherian spectrum.
\end{proof}

\begin{corollary}
	Let $R$ be a ring.
	Then  the following statements are equivalent.
	\begin{enumerate}
		\item $R$ has $w$-Noetherian spectrum.
		\item  $R[X]$ has  $w$-Noetherian spectrum.
		\item
		$R\{X\}$ has  Noetherian spectrum.
	\end{enumerate}	
\end{corollary}
\begin{proof}
	Take $S=\{1\}$ in Theorem \ref{main-2}
\end{proof}

\end{document}